\documentclass{article}
\usepackage{amssymb}
\usepackage{verbatim}
\usepackage{array}
\usepackage{latexsym}
\usepackage{enumerate}
\usepackage{amsmath}
\usepackage{amsfonts}
\usepackage{amsthm}
\usepackage{color}
\usepackage[english]{babel}

\newtheorem{theorem}{Theorem}[section]
\newtheorem{proposition}[theorem]{Proposition}
\newtheorem{lemma}[theorem]{Lemma}

\newtheorem{result}[theorem]{Result}
\newtheorem{openproblem}[theorem]{Open Problem}
\newtheorem{rem}[theorem]{Remark}
\def\cC{\mathcal C}

\def\cQ{\mathcal Q}

\def\cX{\mathcal X}
\def\cY{\mathcal Y}

\def\K{\mathbb{K}}

\def\ord{\mbox{\rm ord}}

\newcommand{\aut}{\mbox{\rm Aut}}

\title{On algebraic curves with many automorphisms in characteristic $p$}
\date{}
\author{Maria Montanucci\thanks{Department of Applied Mathematics and Computer Science, Technical University of Denmark, Kongens Lyngby 2800, Denmark,  {\em email: } marimo@dtu.dk}}

\begin{document}
\maketitle

\begin{center}
\emph{Dedicated to the memory of Elisa Montanucci.}
\end{center}

\begin{abstract}
 Let $\cX$ be an irreducible, non-singular, algebraic curve defined over a field of odd characteristic $p$. Let $g$  and $\gamma$ be the genus and $p$-rank of $\cX$, respectively. The influence of $g$ and $\gamma$ on the automorphism group $Aut(\cX)$ of $\cX$  is well-known in the literature. If $g \geq 2$ then $Aut(\cX)$ is a finite group, and unless $\cX$ is the so-called Hermitian curve, its order is upper bounded by a polynomial in $g$ of degree four (Stichtenoth). In 1978 Henn proposed a refinement of Stichtenoth's bound of degree $3$ in $g$ up to few exceptions, all having $p$-rank zero. In this paper a further refinement of Henn's result is proposed. First, we prove that if an algebraic curve of genus $g \geq 2$ has more than $336g^2$ automorphisms then its automorphism group has exactly two short orbits, one tame and one non-tame, that is, the action of the group is completely known. Finally when $|Aut(\cX)| \geq 900g^2$ sufficient conditions for $\mathcal{X}$ to have $p$-rank zero are provided.
\end{abstract}

\begin{small}



\end{small}

\section{Introduction}

Let $\cX$ be a projective, geometrically irreducible, non-singular algebraic curve defined over an algebraically closed field $\K$ of positive characteristic $p$. Let $\K(\cX)$ be the field of rational functions on $\cX$ (i.e. the function field of $\cX$ over $\K$).
The $\K$-automorphism group $Aut(\cX)$ of $\cX$ is defined as the automorphism group of $\K(\cX)$ fixing $\K$ element-wise. The group $Aut(\cX)$ has a faithful action on the set of points of $\cX$.

By a classical result, $Aut(\cX)$ is finite whenever the genus $g$ of $\cX$ is at least two; see \cite{S} and \cite[Chapter 11; 22--24,31,33]{HKT}.
Furthermore it is known that every finite group occurs in this way, since, for any ground field $\K$ and any finite group $G$, there exists an algebraic curve $\cX$ defined over $\K$ such that $Aut(\cX)\cong G$; see \cite{27,28,38}.

This result raised a general problem for groups and curves, namely, that of determining the finite groups that can be realized as the $\K$-automorphism group of some curve with a given invariant. The most important such invariant is the \emph{genus} $g$ of the curve. In positive characteristic, another important invariant is the so-called $p$-rank of the curve, which is the integer $0 \leq \gamma \leq g$ such that the Jacobian of $\cX$ has $p^\gamma$ points of order $p$. 

Several results on the interaction between the automorphism group, the genus and the $p$-rank of a curve can be found in the literature. 
A remarkable example is the work of Nakajima \cite{Nak} who showed that the value of the $p$-rank deeply influences the order of a $p$-group of automorphisms of $\cX$. He also showed that curves for which the $p$-rank is the largest possible, namely $\gamma=g$, have at most $84(g^2-g)$ automorphisms.

In \cite{H} Hurwitz showed that if $\cX$ is defined over $\mathbb{C}$ then $|Aut(\cX)|\leq 84(g-1)$, which is known as \emph{Hurwitz bound}. This bound is sharp, i.e., there exist algebraic curves over $\mathbb{C}$ of arbitrarily high genus whose automorphism group has order exactly $84(g-1)$. Well-known examples are the Klein quartic and the Fricke-Macbeath curve, see \cite{M}. 
Roquette \cite{Ro} showed that Hurwitz bound also holds in positive characteristic $p$, if $p$ does not divide $|Aut(\cX)|$. 

A general bound in positive characteristic is $|Aut(\cX)| \leq 16g^4$ with one exception: the so-called Hermitian curve. This result is due to Stichtenoth \cite{stichtenoth1973I,stichtenoth1973II}.

The quartic bound $|Aut(\cX)| \leq 16g^4$ was improved by Henn in \cite{henn1978}. Henn's result shows that if $|Aut(\cX)| >8g^3$ then $\cX$ is $\K$-isomorphic to one of the following four curves:

\begin{itemize}
\item the non-singular model of the plane curve $Y^2 + Y = X^{2k+1}$, with $k > 1$ and $p=2$; 
\item the non-singular model of $Y^2 = X^n -X$, where $n = p^h$, $h > 0$ and $p>2$;
\item the Hermitian curve $\mathcal{H}_q: Y^{q+1}=X^q+X$ where $q = p^h$ and $h > 0$;
\item the non-singular model of the Suzuki curve $\mathcal{S}_q: X^{q_0}(X^q + X) = Y^q + Y$, where $q_0 = 2^r$, $r \geq 1$ and $q = 2q_0^2$. 
\end{itemize}

All the above exceptions have $p$-rank zero. This observation raised the following problem.

\begin{openproblem} \label{op}
Is it possible to find a (optimal) function $f(g)$ such that the existence of an automorphism group $G$ of $\cX$ with $|G|\geq f(g)$ implies $\gamma=0$?
\end{openproblem}

Clearly from Henn's result $f(g) \leq 8g^3$. Also $f(g)$ cannot be asymptotically of order less than $g^{3/2}$ as algebraic curves of positive $p$-rank with approximately $g^{3/2}$ automorphisms are known; see for example \cite{KMS}.

Open Problem \ref{op} was already studied in \cite{GK} where a positive answer is given under the additional hypothesis that $g$ is even or that the automorphism group $G$ is solvable.

\begin{theorem}  \label{even} \rm{\cite[Theorem 1.1 and Theorem 1.2]{GK}}
Let $\K$ be an algebraically closed field of odd characteristic $p$ and let $\cX$ be an algebraic curve defined over $\K$.
If $\cX$ has even genus $g \geq 2$ and at least $900g^2$ automorphisms then its $p$-rank $\gamma$ is equal to zero.
If $\cX$ is of arbitrary genus $g \geq 2$ and it has a solvable automorphism group of order at least $84pg^2/(p-2)$ then the $p$-rank of $\cX$ is zero.
\end{theorem}

In this paper we analyze large automorphism groups of curves of arbitrary genus $g \geq 2$ giving a partial answer to Open Problem \ref{op}. The following theorem summarizes our main results.

\begin{theorem} \label{maint}
Let $G$ be an automorphism group of an algebraic curve $\cX$ defined over a field of odd characteristic $p$. Denote with $g \geq 2$ and $\gamma$ the genus and the $p$-rank of $\cX$ respectively.
\begin{enumerate}
\item If $|G| > 24g^2$ then either $G$ has a unique (non-tame) short orbit or it has exactly two short orbits, one tame and one non-tame.
\item If $G$ has exactly one non-tame short orbit then $|G| \leq 336g^2$.
\item If $|G| \geq 60g^2$ and $G$ has exactly one non-tame short orbit then $\gamma$ is positive and congruent to zero modulo $p$.
\item If $|G| \geq 900g^2$ then $G$ has exactly one non-tame short orbit $O_1$ and one tame short orbit $O_2$.  If $\cX/G_P^{(1)}$ is rational for $P \in O_1$ and the stabilizer $G_{P,R}$ with $R \in O_1 \setminus \{P\}$ is either a $p$-group or a prime-to-p group then $\gamma$ is zero.

 \end{enumerate}
\end{theorem}

Note that this theorem implies that whenever a quadratic bound like $|Aut(\mathcal{X})| >336g^2$ holds, then the action of the group is completely known, having one tame and one non-tame short orbits.

The paper is organized as follows. In Section 2 some preliminary results on automorphism groups of algebraic curves in positive characteristic are recalled. In Section 3 Parts 1-3 of Theorem \ref{maint} are proven, while Part 4 is the main object of Section 4.

\section{Preliminary Results}
In this paper, $\cX$ stands for a (projective, geometrically irreducible, non-singular) algebraic curve of genus $g=g(\cX) \ge 2$ defined over an algebraically closed field $\K$ of odd characteristic $p$. Let $\aut(\cX)$ be the group of all automorphisms of $\cX$. The assumption $g(\cX)\geq 2$ ensures that $\aut(\cX)$ is finite. However the classical Hurwitz bound
 $|\aut(\cX)| \leq 84(g(\cX)-1)$ for complex curves fails in positive characteristic, and there exist four families of curves satisfying $|\aut(\cX)|\geq 8g(\cX)^3$; see \cite{stichtenoth1973II}, Henn \cite{henn1978}, and also \cite[Section 11.12]{HKT}. 

For a subgroup $G$ of $\aut(\cX)$, let $\bar \cX$ denote a non-singular model of $\K(\cX)^G$, that is,
a (projective non-singular geometrically irreducible) algebraic curve with function field $\K(\cX)^G$, where $\K(\cX)^G$ consists of all elements of $\K(\cX)$ fixed by every element in $G$. Usually, $\bar \cX$ is called the
quotient curve of $\cX$ by $G$ and denoted by $\cX/G$. The field extension $\K(\cX)|\K(\cX)^G$ is Galois of degree $|G|$.

Let $\Phi$ be the cover of $ \Phi: \cX \rightarrow \bar{\cX}$ where $\bar{\cX}=\cX/G$. A point $P\in\cX$ is a ramification point of $G$ if the stabilizer $G_P$ of $P$ in $G$ is nontrivial; the ramification index $e_P$ is $|G_P|$; a point $\bar{Q}\in\bar{\cX}$ is a branch point of $G$ if there is a ramification point $P\in \cX$ such that $\Phi(P)=\bar{Q}$; the ramification (branch) locus of $G$ is the set of all ramification (branch) points. The $G$-orbit of $P\in \cX$ is the subset
$o=\{R\mid R=g(P),\, g\in G\}$ of $\cX$, and it is {\em long} if $|o|=|G|$, otherwise $o$ is {\em short}. For a point $\bar{Q}$, the $G$-orbit $o$ lying over $\bar{Q}$ consists of all points $P\in\cX$ such that $\Phi(P)=\bar{Q}$. If $P\in o$ then $|o|=|G|/|G_P|$ and hence $\bar{\cQ}$ is a branch point if and only if $o$ is a short $G$-orbit. It may be that $G$ has no short orbits. This is the case if and only if every non-trivial element in $G$ is fixed--point-free on $\cX$, that is, the cover $\Phi$ is unramified. On the other hand, $G$ has a finite number of short orbits.

For a non-negative integer $i$, the $i$-th ramification group of $\cX$
at $P$ is denoted by $G_P^{(i)}$ (or $G_i(P)$ as in \cite[Chapter
IV]{serre1979})  and defined to be
$$G_P^{(i)}=\{\alpha \in G_P \mid \ord_P(\alpha(t)-t)\geq i+1\}, $$ where $t$ is a uniformizing element (local parameter) at
$P$. Here $G_P^{(0)}=G_P$.

Let $\bar{g}$ be the genus of the quotient curve $\bar{\cX}=\cX/G$. The \emph{Hurwitz
genus formula} \cite[Theorem 3.4.13]{Sti} gives the following equation
    \begin{equation}
    \label{eq1}
2g-2=|G|(2\bar{g}-2)+\sum_{P\in \cX} d_P,
    \end{equation}
    where the different $d_P$ at $P$ is given by
\begin{equation}
\label{eq1bis}
d_P= \sum_{i\geq 0}(|G_P^{(i)}|-1),
\end{equation}
see \cite[Theorem 11.70]{HKT}.

Let $\gamma$ be the $p$-rank of $\cX$, and let $\bar{\gamma}$ be the $p$-rank of the quotient curve $\bar{\cX}=\cX/G$.
A formula relating $\gamma$ and $\bar \gamma$ is known whenever $G$ is a $p$-group.
Indeed if $G$ is a $p$-group, the \emph{Deuring-Shafarevich formula} states that
\begin{equation}
    \label{eq2deuring}
    %modifica 23 marzo
\gamma-1={|G|}(\bar{\gamma}-1)+\sum_{i=1}^k (|G|-\ell_i),
    \end{equation}
    %%modifica 2 marzo 2009

where $\ell_1,\ldots,\ell_k$ are the sizes of the short orbits of $G$; see \cite{sullivan1975} or \cite[Theorem 11,62]{HKT}.

A subgroup of $\aut(\cX)$ is  a prime-to-$p$ group (or a $p'$-subgroup) if its order is prime to $p$. A subgroup $G$ of $\aut(\cX)$ is {\em{tame}} if the $1$-point stabilizer of any point in $G$ is $p'$-group. Otherwise, $G$ is {\em{non-tame}} (or {\em{wild}}). By \cite[Theorem 11.56]{HKT}, if $|G|>84(g(\cX)-1)$ then $G$ is non-tame. 

An orbit $o$ of $G$ is {\em{tame}} if $G_P$ is a $p'$-group for $P\in o$. 
%The stabilizer $G_P$ of a point $P\in \cX$ in $G$ is a semidirect product $G_P=Q_P\rtimes U$ where the normal subgroup $Q_P$ is a $p$-group while the complement $U$ is a cyclic prime to $p$ group; see \cite[Theorem 11.49]{HKT}.
The following lemma gives a strong restriction to the action of the Sylow $p$-subgroup of the stabilizer of a point $P \in \cX$ when $\gamma=\gamma(\cX)=0$.

\begin{result} \rm{\cite[Lemma 11.129 ]{HKT}} \label{trivialintersectionset}
If $\gamma(\cX)=0$ then every element of order $p$ in $\aut(\cX)$ has exactly one fixed point on $\cX$.
\end{result}

Bounds for the order of tame automorphism groups fixing a point are known; see \cite[Theorem 11.60]{HKT}.

\begin{result} \label{tameorder}
 Let $\cX$ be an irreducible curve of genus $g > 0$, and let $G_P$ be a $\K$-automorphism group of $\cX$ fixing a point $P$. If the order $n$ of $G_P$ is prime to $p$, then $n \leq 4g + 2$. 
\end{result}

%For an algebraic curve $\cX$ defined over a finite field $\mathbb{F}_q$, denote by $N_n=|\cX(\mathbb{F}_{q^n})|$. If the numbers $N_n$ are known for $n=1,\ldots,r$ and sufficiently large $r$, then the coefficients of the $L$-polynomial of $\cX$ can be calculated as follows.
%
%\begin{proposition} \label{Lpol} \rm{\cite[Proposition 9.12]{HKT}}
%Let $L(t)=\sum_{i=0}^{2g} a_it^i$ be the $L$-polynomial of $\cX$, and let $s_n=N_n-(q^n+1)$.
%\begin{itemize}
%\item[(i)] The logarithmic derivative of $L(t)$ is given by
%$$\frac{dL(t)/dt}{L(t)}=\sum_{n=1}^{\infty} s_n t^{n-1}.$$
%\item[(ii)] For $n=1,\ldots,g$,
%\begin{equation} \label{coeffLpol}
%na_n=s_n+s_{n-1}a_1+\ldots+s_1a_{n-1}.
%\end{equation}
%In particular, from $N_1,\ldots,N_g$, the coefficients of $L(t)$ can be calculated using \eqref{coeffLpol} and the equations $a_{2g-n}=q^{g-n} a_n$, for $n=1,\ldots,g$.
%\end{itemize}
%\end{proposition}
%
%Write $L(t)=1+a_1t+\ldots+a_it^i+\ldots+t^{2g}$ and denote by $\tilde{\gamma}(\cX)$ the highest index $i$ with $1 \leq i \leq g$ for which $a_i \not\equiv 0 \pmod p$.
%
%\begin{theorem} \rm{\cite[Theorem 9.75]{HKT}} \label{prank}
%The index $\tilde{\gamma}(\cX)$ equals the $p$-rank $\gamma(\cX)$ of $\cX$.
%\end{theorem}

Strong restrictions for the short orbits structure of automorphism groups of algebraic curves are known when the Hurwitz bound $84(g(\cX)-1)$ fails. In particular, the following theorem ensures that automorphism groups for which the Hurwitz bound is not satisfied have at most three short orbits.

\begin{theorem} \rm{\cite[Theorem 11.56]{HKT}} \label{Hurwitz}
 Let $\cX$ be an irreducible curve of genus $g \geq  2$ defined over a field $\K$ of characteristic $p$. 
\begin{itemize}
\item If $G$ is a $\K$-automorphism group of $\cX$, then the Hurwitz's upper bound $|G| \leq 84(g-1)$ holds in general with exceptions occurring only if $p>0$. 
\item If $p>0$ then exceptions can only occur when the fixed field $\K(\cX)^G$ is rational and $G$ has at most three short orbits as follows: 
\begin{enumerate}
\item exactly three short orbits, two tame and one non-tame, with $p \geq 3$;
\item  exactly two short orbits, both non-tame; 
\item  only one short orbit which is non-tame; 
\item  exactly two short orbits, one tame and one non-tame.
\end{enumerate}
\end{itemize}
\end{theorem}

\begin{theorem} \rm{\cite[Theorem 11.116 and Theorem 11.125]{HKT}} \label{bounds}
 Let $\cX$ be an algebraic curve of genus $g \geq  2$ defined over an algebraically closed field $\K$ of positive characteristic $p$. If $G$ is an automorphism group of $\cX$ with $|G| > 84(g-1)$, then an upper bound for the order of $G$ in Cases 1,2 of Theorem \ref{Hurwitz} is given by:
\begin{enumerate}
\item $|G| <24g^2$,
\item $|G|<16g^2$,
\end{enumerate}
respectively. 
If $G$ satisfies Case 3 of Theorem \ref{Hurwitz} then $|G|<8g^3$. If $G$ satisfies Case $4$ in Theorem \ref{Hurwitz} then $|G|<8g^3$ unless one of the following cases occurs up to isomorphism over $\K$:
\begin{itemize}
\item $p = 2$ and $\cX$ is the non-singular model of the plane curve $Y^2 + Y = X^{2k+1}$, with $k > 1$; 
\item $p > 2$ and $\cX$ is the non-singular model of $Y^2 = X^n -X$, where $n = p^h$ and $h > 0$;
\item $\cX$ is the Hermitian curve $\mathcal{H}_q: Y^{q+1}=X^q+X$ where $q = p^h$ and $h > 0$;
\item  $\cX$ is the non-singular model of the Suzuki curve $\mathcal{S}_q: X^{q_0}(X^q + X) = Y^q + Y$, where $q_0 = 2^r$, $r \geq 1$ and $q = 2q_0^2$. 
\end{itemize}
Furthermore, all the above algebraic curves have $p$-rank zero.
\end{theorem}

Theorem \ref{bounds} shows in particular that a quadratic bound on $|G|$ in Case 4 of Theorem \ref{Hurwitz} is not possible. However nothing is known in Case 3 of Theorem \ref{Hurwitz}. One of the first aims of this paper is to show that actually a quadratic bound with respect to the genus can be found also in this case.

\begin{rem}
Examples of algebraic curves of genus $g$ with approximately $g^2$ automorphisms satisfying Case 4 in Theorem \ref{Hurwitz} are known. Given a prime power $q$, the GK curve $\cC$  is given by the affine model,
$$\cC: \begin{cases} y^{q+1}=x^q+x, \\ z^{q^2-q+1}=y^{q^2}-y, \end{cases}$$ 
see \cite{GKcurve}.
The curve $\cC$ has genus $g(\cC)=(q^5-2q^3+q^2)/2$ and it is $\mathbb{F}_{q^6}$-maximal. The automorphism group of $\cC$ is defined over $\mathbb{F}_{q^6}$ and has order $q^3(q^3+1)(q^2-1)(q^2-q+1) \sim 4g(\cC)^2$. The set $\cC(\mathbb{F}_{q^6})$ of the $\mathbb{F}_ {q^6}$-rational points of $\cC$ splits into two orbits under the action of $Aut(\cC)$: $O_1=\cC(\mathbb{F}_{q^2})$  and $O_2=\cC(\mathbb{F}_{q^6}) \setminus \cC(\mathbb{F}_{q^2})$. The orbit $O_1$ is non-tame while $O_2$ is tame. Case 4 is indeed satisfied, see \cite[Theorem 7]{GKcurve}.

Other two examples are the cyclic extensions of the Suzuki and Ree curves constructed in \cite{Skabe}. Again the order of the automorphism group of these curves is approximately $4g^2$ and it satisfies Case 4 in Theorem \ref{Hurwitz}; see \cite{GMQZ}.
All the examples written in this remark have $p$-rank zero.
\end{rem}

 In order to give our partial answer to Open Problem \ref{op} the following lemmas from \cite{GK} will be used.

\begin{lemma}\rm{\cite[Lemma 4.1 and Remark 4.3]{GK}} \label{terribile}
Let $\cX$ be an algebraic curve of genus $g \geq  2$ defined over an algebraically closed field $\K$ of odd characteristic $p$. 
Let $H$ be an automorphism group of $\cX$ with a normal Sylow $d$-subgroup $Q$ of odd order. Suppose that a complement $U$ of $Q$ in $H$ is cyclic and that $N_H(U) \cap Q=\{1\}$. If 
$$|H| \geq 30(g-1),$$
then $d=p$ and $U$ is cyclic. Moreover, the quotient curve $\bar{\cX}=\cX/Q$ is rational and either 
\begin{enumerate}
\item $\cX$ has positive $p$-rank, $Q$ has exactly two (non-tame) short orbits, and they are also the only short orbits of $H$; or
\item $\cX$ has zero $p$-rank and $H$ fixes a points.
\end{enumerate}
If $d=p$ is assumed then the hypothesis $N_H(U) \cap Q=\{1\}$ is unnecessary.
\end{lemma}

\begin{lemma} \rm{\cite[Lemma 4.8]{GK}}\label{terribile2}
Let $G$ be an automorphism group of an algebraic curve $\cX$ of genus $g \geq 2$ defined over a field of odd characteristic $p$. Suppose that 
\begin{enumerate}
\item $|G| \geq 16g^2$;
\item any two distinct Sylow $p$-subgroups of $G$ have trivial intersection;
\item $G$ has a Sylow $d$-subgroup $Q$ for which its normalizer $N_G(Q)$ contains a subgroup $H$ satisfying the hypotheses of Lemma \ref{terribile};
\end{enumerate} 
then $\cX$ has zero $p$-rank.
\end{lemma}

%\begin{lemma} \rm{\cite[Proposition 3.1]{GK}}\label{terribile3}
%Let $\cX$ be an algebraic curve of genus $g \geq 2$ with a solvable automorphism group $G$ such that 
%$$|G|>84\frac{p}{p-2}g^2,$$
%then $\cX$ has zero $p$-rank.
%\end{lemma}

The following result provides a list of known and useful properties of automorphism groups of algebraic curves in positive characteristic.

\begin{result} \label{list}
Let $\cX$ be an algebraic curve of genus $g \geq 2$ defined over a field of characteristic $p \geq 3$.
Let $G$ be an automorphism group of $\cX$.
\begin{enumerate}
\item \em{\cite[Theorem 1]{Nak}} If $G$ is a $p$-group and $\gamma(\cX)>0$ then $|G| \leq p(g-1)/(p-2)$. If $\gamma(\cX)=0$ then $|G| \leq \max\{g, 4pg^2/(p-1)^2\}$.
\item \rm{\cite[Theorem 11.60 and Theorem 11.79]{HKT}} If $G$ is abelian then $|G| \leq 4g+4$.
\item \rm{\cite[Theorem 11.78]{HKT}} If $P \in \cX$ is such that the quotient curve $\cX/G_P^{(1)}$ is not rational then $|G_P^{(1)}| \leq g$.
%\item \rm{\cite[Lemma 11.129]{HKT}} If $\gamma(\cX)=0$ then every $p$-element in $G$ has exactly one fixed point on $\cX$.
\item \rm{\cite[Lemma 11.44 (e)]{HKT}} For $P \in \cX$, the stabilizer $G_P$ of $P$ in $G$ is a semidirect product $G_P=G_P^{(1)} \rtimes U$ where $U$ is a cyclic $p^\prime$-group and $G_P^{(1)}$ is the Sylow $p$-subgroup of $G_P$. In particular $G_P$ is solvable.
\item \rm{\cite[Theorem 11.14]{HKT}} If $\cX$ is rational and $G$ is cyclic and tame then $G$ has exactly two fixed points and no other short orbits on $\cX$.
\end{enumerate}
\end{result}

An essential ingredient from group theory that we will use in the proof of Theorem \ref{maint} is the complete list of finite $2$-transitive permutation groups, see \cite[Tables 7.3 and 7.4]{cameron}.

A well known theorem of Burnside \cite[Theorem 4.3]{cameron} states that every finite $2$-transitive group is either almost simple or affine. 
Finite affine $2$-transitive groups are those having an elementary abelian regular normal socle, while almost simple $2$-transitive groups are those having a simple socle $N$.

Tables \ref{table1} and \ref{table2} list finite 2-transitive groups according to the two aforementioned categories.
Recall that the degree of a $2$-transitive permutation group is the cardinality of the set on which the group acts $2$-transitively.

\begin{table}[h!]
  \begin{center}
    \caption{Affine 2-transitive groups}
    \label{table1}
    \begin{tabular}{l|c|c|r} % <-- Alignments: 1st column left, 2nd middle and 3rd right, with vertical lines in between
     \textbf{Case} & \textbf{Degree} & \textbf{$\bf G_P$} & \textbf{Condition}\\
      \hline
     1. & $q^d$ & $SL(d,q) \leq G_P \leq \Gamma L(d,q)$ & \\
      2. & $q^{2d}$ & $Sp(d,q) \trianglelefteq G_P$ & $d \geq 2$\\
      3. & $q^6$ & $G_2(q) \trianglelefteq G_P$ & $q$ even\\
     4. & $q$ & $(2^{1+2} \rtimes 3)=SL(2,3) \trianglelefteq G_P$ & $q=5^2,7^2,{11}^2,23^2$\\
 5. & $q$ & $2^{1+4} \trianglelefteq G_P$ & $q=3^4$\\
 6. & $q$ & $SL(2,5) \trianglelefteq G_P$ & $q={11}^2,{19}^2,{29}^2,{59}^2$\\
7. & $2^4$ & $A_6$& \\
8. &  $2^4$ & $A_7$ & \\
 9. & $2^6$ & $PSU(3,3)$ & \\
 10. & $3^6$ & $SL(2,12)$ & \\
\end{tabular}
  \end{center}
\end{table}

\begin{table}[h!]
  \begin{center}
    \caption{Almost Simple 2-transitive groups}
    \label{table2}
    \begin{tabular}{l|c|c|c|r} 
      \textbf{Case} & \textbf{Degree} & \textbf{Condition} & $ \bf N$ & $ \bf max|G/N|$\\
      \hline
1. &   $ n$ & $n \geq 5$ & $A_n$ & $2$ \\
2. & $ (q^d-1)/(q-1)$ & $d \geq 2$, $(d,q) \ne (2,2),(2,3)$ & $PSL(d,q)$ & $(d,q-1)$ \\
3. & $2^{2d-1}+2^{d-1}$ & $d \geq 3$ & $Sp(2d,2)$ & $1$ \\
4. & $2^{2d-1}-2^{d-1}$ & $d \geq 3$ & $Sp(2d,2)$ & $1$ \\
5. & $q^3+1$ & $q \geq 3$ & $PSU(3,q)$ & $(3,q+1)$ \\
6. & $q^2+1$ & $q=2^{2d+1}>2$ & $Sz(q)$ & $2d+1$ \\
7. & $q^3+1$ & $q=3^{2d+1}>3$ & $Ree(q)$ & $2d+1$ \\
8. & $11$ & & $PSL(2,11)$ & $1$ \\
9. & $11$ & & $M_{11}$ & $1$ \\
10. & $12$ &  & $M_{11}$ & $1$ \\
11. & $12$ &  & $M_{12}$ & $1$ \\
12. & $15$ & & $A_7$ & $1$ \\
13. & $22$ &  & $M_{22}$ & $2$ \\
14. & $23$ & & $M_{23}$ & $1$ \\
15. & $24$ & & $M_{24}$ & $1$ \\
16. & $28$ &  & $PSL(2,8)$ & $3$ \\
17. & $176$ & & $HS$ & $1$ \\
18. & $276$ & & $C_{o_3}$ & $1$ \\
\end{tabular}
  \end{center}
\end{table}

The list of finite $2$-transitive permutation groups can be refined if the stabilizer of two points is cyclic.

\begin{theorem} \label{KOS2t} \rm{\cite[Theorem 1.1]{KOS}}
Let $G$ be a finite, $2$-transitive permutation group on a set $\cX$. Suppose that the stabilizer $G_{P,Q}$ of two distinct points $P,Q \in \cX$ is cyclic, and that $G$ has no regular normal subgroups. Then $G$ is one of the following groups in its usual $2$-transitive permutation representation: $PSL(2, q)$, $PGL(2, q)$, $Sz(q)$, $PSU(3, q)$, $PGU(3, q)$ or a group of Ree type.
\end{theorem}

\begin{theorem} \label{BWrn} \rm{\cite[Theorem 1.7.6]{BW}}
Let $G$ be a $2$-transitive group on a set $O$. If $G$ has a regular normal subgroup $N$, then $N$ is an elementary abelian $d$-group where $d$ is a prime and $|O|=d^n$ for some $n \geq 1$.
\end{theorem}

With all the ingredients introduced in Section 2 we can proceed with the proof of Theorem \ref{maint}. Doing so, we can assume that $G$ is an automorphism group of an algebraic curve $\cX$ of genus $g=g(\mathcal{X}) \geq 2$, $p$-rank $\gamma=\gamma(\mathcal{X})$, and such that $|G|>84(g-1)$. By Theorems 2.3 and 2.4 unless $|G| \leq 24g^2$, $G$ has either exactly one short orbit (Case 3 in Theorem 2.3) or exactly two short orbits, one tame and one non-tame (Case 4 in Theorem 2.4).
We analyze these two cases separately in Sections 3 and 4 respectively.

\section{$G$ satisfies Case $3$ of Theorem \ref{Hurwitz}}

In this section $G$ stands for an automorphism group of an algebraic curve $\cX$ of genus $g \geq 2$ defined over a field $\K$ of odd characteristic $p$ satisfying Case 3 of Theorem \ref{Hurwitz}.
We denote by $O$ the only short orbit of $G$ on $\cX$. 

We start with the following direct consequence of the Hurwitz genus formula.

\begin{lemma}\label{1o}
If an automorphism group $G$ of an algebraic curve $\cX$ of genus $g \geq 2$ has exactly one short orbit $O$ then the size of $O$ divides $2g-2$. 
In particular this holds true whenever $G$ satisfies Case 3 in Theorem \ref{Hurwitz}.
\end{lemma}

\begin{proof}
Let $\bar{g}$ be the genus of the quotient curve $\cX/G$. If the stabilizer $G_P$ of a point $P \in O$ is tame, then the Hurwitz genus formula \eqref{eq1} reads as follows
$$2g-2 = |G|(2\bar{g}-2) +|G|-|O|.$$
As from the Orbit stabilizer Theorem $|G| =|G_P||O|$, the result follows. 

In the non-tame case, in Equation \eqref{eq1} we have $d_P = d_Q$  for $P,Q \in O$ and $d_R = 0$ for $R \not\in O$. Hence $|O|$ divides $\sum_{P \in \cX} d_P$, and the result follows as in the tame case from Equation \eqref{eq1}.
\end{proof}

We now move to the proof of Part 2 in Theorem \ref{maint}.

\begin{theorem} \label{main3}
If $|G|\geq 60g^2$ then $\gamma$ is congruent to zero modulo $p$.
\end{theorem}

\begin{proof}
Let $P \in O$ and let $G_P$ be the stabilizer of $P$ in $G$.
From Item 4 of Result \ref{list}, we can write $G_P=G_P^{(1)} \rtimes U$ where $U$ is tame and cyclic. Denote by $\cY_1$ the quotient curve $\cX/G_P^{(1)}$. 
We distinguish two cases.
\begin{itemize}
\item \textbf{Case 1: $O=\{P\}$.} If $\cY_1$ is not rational then Item 3 of Result \ref{list} and Result \ref{tameorder} imply that $|G|=|G_P| \leq g(4g+2)<60g^2$, a contradiction.

 Hence $g(\cY_1)=0$. The factor group $U_1=G_P/G_P^{(1)} \cong U$ is a prime-to-$p$ cyclic automorphism group of $\cY_1$ fixing the point $P_1$ lying below $P$ in the cover $\cX| \cY_1$. 
If $|U|=1$ then a contradiction to $|G|<60g^2$ is obtained from Result 2.8 part 1. Hence we can assume $U_1$ non-trivial.
From  Item 5 of Result \ref{list} the group $U_1$ has another fixed point, say $Q_1$, on $\cY_1$. If $Q$ denotes a point of $\cX$ lying above $Q_1$ in $\cX|\cY_1$ then $U$ acts on the $G_P^{(1)}$-orbit $\Delta$ containing $Q$ and since $|U|$ is prime-to-$p$ we get that $U$ has at least another fixed point  on $\cX$, say $R$, which is also contained in $\Delta$. This implies that the $G$-orbit containing $R$ is short as the stabilizer of $R$ in $G$ is non-trivial. This is not possible since by hypothesis $O=\{P\}$ is the only short orbit of $G$. 
\item \textbf{Case 2: $O \supset \{P\}$.} From the Orbit stabilizer Theorem $|G|=|O||G_P|$. Also from  Lemma \ref{1o} we can write $2g-2=k|O|$ for some $k \geq 1$. Then
$$(2g-2)^2=k^2|O|^2=k^2 \frac{|G|^2}{|G_P|^2},$$
and hence since $|G|\geq 60g^2$,
$$|G_P|=k^2\frac{|G|^2}{|G_P|(2g-2)^2} = \frac{k^2|G|}{(2g-2)^2} \cdot \frac{|G|}{|G_P|}>\frac{|G|}{|G_P|}.$$
 
From $|O| \leq 2g-2$ we have
$$2g|G_P| > (2g-2)|G_P| \geq |O||G_P| = |G|\geq 60g^2$$
and hence $|G_P| \geq 30g>30(g-1)$.  From Lemma \ref{terribile}  either $\gamma(\cX)=0$ or $\gamma(\cX)>0$ and in the latter case $G_P^{(1)}$ has exactly two (non-tame) short orbits which are also the only short orbits of $G_P$.
Assume that $\gamma(\cX)>0$ and denote by  $O_1=\{P\}$ and $O_2$  the two short orbits of $G_P^{(1)}$ and $G_P$, with $|O_2|=p^i$ and $i \geq 0$.

If $i=0$ we have that $G_P^{(1)}$ and $G_P$ fix exactly another point $R \in \cX \setminus \{P\}$ and have no other short orbits.
Since $O_1$ and $O_2$ are contained in $O$, we can write $|O|=2+h|G_P|$ for some $h \geq 0$. If $h=0$ then $|O|=2$ and hence from Item 1 of Result \ref{list} and Result \ref{tameorder}, $|G|=|G_P||O| =2|G_P| \leq 2p(g-1)(4g+2)/(p-2) \leq 30g^2$, a contradiction. If $h \geq 1$ then $|G_P|<|O| \leq 2g-2$. Since $|G_P|\geq 30g$ we get a contradiction.

This shows that necessarily $i \geq 1$.
From the Deuring-Shafarevic formula \eqref{eq2deuring} applied to $G_P^{(1)}$ one has
$$\gamma-1=|G_p^{(1)}|(0-1)+|G_p^{(1)}|-1+|G_p^{(1)}|-p^i,$$
and hence $\gamma$ is congruent to zero modulo $p$.
\end{itemize}
\end{proof}

We now prove that actually the case $\gamma=0$ in Theorem \ref{main3} cannot occur.

\begin{theorem} \label{second}
 Let $\cX$ be an irreducible curve of genus $g \geq  2$. If $G$ is an automorphism group of $\cX$ satisfying Case 3 of Theorem \ref{Hurwitz} then either $|G| < 60g^2$ or $\gamma=\gamma(\cX)$ is positive and congruent to zero modulo $p$.
\end{theorem}  

\begin{proof}
Suppose that $G$ satisfies Case 3 in Theorem \ref{Hurwitz} and $|G| \geq 60g^2$. For $P \in O$ we can write $G_P=G_P^{(1)} \rtimes U$, where $U$ is tame and cyclic from Item 4 of Result \ref{list}. Then from Theorem \ref{main3} we have either $\gamma(\cX)=0$ or $\gamma(\cX)$ positive and congruent to zero modulo $p$. Suppose by contradiction that $\gamma(\cX)=0$.
% From Lemma \ref{1o} and  the Orbit stabilizer Theorem, 
%$$60g^2 \leq |G|=|G_P||U| \leq |G_P|(2g-2).$$
%Hence 
%$$|G_P| \geq \frac{60g^2}{2g-2} >30(g-1).$$ 
Recall that $\cX/G_P^{(1)}$ is rational from Lemma \ref{terribile} as $|G_P| \geq 30g$.

Write $|G_P^{(1)}|=p^h$ with $h \geq 1$. By the Hurwitz genus formula applied with respect to $G_P^{(1)}$ one has,

$$2g-2=-2|G_P^{(1)}|+d_P = -2|G_P^{(1)}| + 2(|G_P^{(1)}| -1)+\sum_{i\geq 2}(|G_P^{(i)}|-1)=\sum_{i\geq 2}(|G_P^{(i)}|-1)-2,$$
since $G_P^{(1)}$ has exactly $\{P\}$ as its unique short orbit from Result \ref{trivialintersectionset}.
On the other hand, recalling that $\cX/G$ is rational, the Hurwitz genus formula applied with respect to $G$ gives
$$2g-2=-2|G|+|O|\bigg(|G_P|-1+|G_P^{(1)}|-1+\sum_{i\geq 2}(|G_P^{(i)}|-1)\bigg)$$
and hence
$$2g-2=-|G|+|O|(|G_P^{(1)}|+2g-2).$$
Therefore,
\begin{equation} \label{p1}
|O|(|G_P|-|G_P^{(1)}|)=|O||G_P^{(1)}|(|U|-1)=(|O|-1)(2g-2)<(2g-2)|O|.
\end{equation}
If $|U|=1$ then $G_P=G_P^{(1)}$ and $|O||G_P^{(1)}|(|U|-1)=(|O|-1)(2g-2)$ implies that $|O|=1$ since $g \geq 2$. Hence $|G| = |G_P^{(1)}| \leq 4g^2$ from  Item 1 of Result \ref{list}; a contradiction.

Since $|U| \geq 2$ we get from Equation \eqref{p1},
$$\frac{|G_P|}{2} \leq |G_P^{(1)}|(|U|-1)<2g-2$$
and $|G|=|O||G_P| <(2g-2)^2<4g^2$; a contradiction.
\end{proof}

This proves Item 2 in Theorem \ref{maint}. Our next goal is to show that up to increasing the value of the constant $c=60$ one can give a complete answer to Open Problem \ref{op} when $G$ satisfies Case 3 in Theorem \ref{Hurwitz}. This will prove Item 3 in Theorem \ref{maint} and show that curves with at least $336g^2$ automorphism have a very precise short orbits structure.

\begin{proposition}
Let $\cX$ be an irreducible curve of genus $g \geq  2$. If $G$ is an automorphism group of $\cX$ satisfying Case 3 of Theorem \ref{Hurwitz}. Then $|G| \leq 336 g^2$.
\end{proposition}

\begin{proof}
Assume by contradiction that there exists an algebraic curve $\cX$ together with an automorphism group $G$ satisfying Case 3 in Theorem \ref{Hurwitz} with $|G| >336 g^2$. We choose $\cX$ to be of minimal genus, that is, if $g^\prime < g$ then an algebraic curve together with an automorphism group $G^\prime$ satisfying Case 3 in Theorem \ref{Hurwitz} with $|G^\prime| >336 {g^\prime}^2$ does not exist. 

The first part of the proof is similar to the one of Theorem \ref{main3}. Let $P \in O$, where $O$ denotes the only short orbit of $G$. Let $G_P^{(1)}$ be the Sylow $p$-subgroup of the stabilizer $G_P$ of $P$ in $G$ and denote by $\cY_1$ the quotient curve $\cX/G_P^{(1)}$. 
From Item 4 of Result \ref{list}, we can write $G_P=G_P^{(1)} \rtimes U$ where $U$ is tame and cyclic.
We distinguish two subcases.
\begin{itemize}
\item \textbf{Case 1: $O=\{P\}$.} If $\cY_1$ is not rational then Item 3 of Result \ref{list} and Result \ref{tameorder} imply that $|G|=|G_P| \leq g(4g+2)<336g^2$, a contradiction.

 Hence $g(\cY_1)=0$. The factor group $U_1=G_P/G_P^{(1)} \cong U$ is a prime-to-$p$ cyclic automorphism group of $\cY_1$ fixing the point $P_1$ lying below $P$ in the cover $\cX| \cY_1$. As before we can assume $|U|>1$. From Item 5 of Result \ref{list}, the group $U_1$ has another fixed point, say $Q_1$, on $\cY_1$. If $Q$ denotes a point of $\cX$ lying above $Q_1$ in $\cX|\cY_1$ then $U$ acts on the $G_P^{(1)}$-orbit $\Delta$ containing $Q$ and since $|U|$ is prime-to-$p$ we get that $U$ has at least another fixed point, say $R$, on $\cX$ which is contained in $\Delta$. This implies that the $G$-orbit containing $R$ is short as the stabilizer of $R$ in $G$ is non-trivial. This is not possible since by hypothesis $O=\{P\}$ is the only short orbit of $G$. 
\item \textbf{Case 2: $O \supset \{P\}$.} From the Orbit stabilizer Theorem $|G|=|O||G_P|$ and $|O|$ divides $2g-2$ from Lemma \ref{1o}. Write $2g-2=k|O|$ for some $k \geq 1$. Then
$$(2g-2)^2=k^2|O|^2=k^2 \frac{|G|^2}{|G_P|^2},$$
and hence since $|G|>336g^2$,
$$|G_P|=k^2\frac{|G|^2}{|G_P|(2g-2)^2} = \frac{k^2|G|}{(2g-2)^2} \cdot \frac{|G|}{|G_P|}>\frac{|G|}{|G_P|}.$$

From $|O| \leq 2g-2$ we have
$$2g|G_P| > |G_P|(2g-2) \geq |G_P||O| = |G|>336g^2$$
and hence $|G_P|>30(g-1)$.  From Lemma \ref{terribile}  either $\gamma(\cX)=0$ or $\gamma(\cX)>0$ and in the latter case $G_P^{(1)}$ has exactly two (non-tame) short orbits, and they are also the only short orbits of $G_P$.
The claim follows by showing that the case $\gamma(\cX)>0$ cannot occur since Theorem \ref{second} gives that the case $\gamma(\cX)=0$ cannot occur either. 

Assume by contradiction that $\gamma(\cX)>0$ and denote by  $O_1=\{P\}$ and $O_2$, with $|O_2|=p^i$, $i \geq 0$ the two short orbits of $G_P^{(1)}$ and $G_P$.
We have the following two possibilities.
\begin{itemize}

\item \textbf{Subcase 2.I: $O \ne O_1 \cup O_2$.} Let $R \in O \setminus (O_1 \cup O_2)$. Then the orbit $\Delta$ of $G_P$ containing $R$ is long and $\Delta \subset O$. Hence $|G_P|=|\Delta| <|O| \leq (2g-2)$. 
This implies that $|G|=|G_P||O| < (2g-2)(2g-2)<4g^2$, a contradiction.

\item \textbf{Subcase 2.II: $O = O_1 \cup O_2$.} Denote by $K$ the kernel of the permutation representation of $G$ on $O$. Since $O=\{P\} \cup O_2$, $G$ acts $2$-transitively on $O$. Then $G/K$ is isomorphic to one of the finite $2$-transitive permutation groups listed in Tables \ref{table1} and \ref{table2}. The claim will follow with a case-by-case analysis. 
\begin{itemize}
\item Suppose that $\bar G=G/K$ is one of the finite affine $2$-transitive groups in Table \ref{table1}. The first three cases cannot occur since $ {\bar G}_P$ is solvable. This follows from the fact that $G_P$ is solvable from Item 4 of Result \ref{list} and hence every factor group of $G_P$ is solvable as well. Cases 4, 5, 6 and 10 can be excluded as $|O|$ must be even. In the remaining cases either $|O|=2^4$ or $|O| = 2^6$. Since $|O|-1$ must be a prime power both the cases can be excluded.

\item Suppose that $\bar G=G/K$ is one of the finite almost simple $2$-transitive groups in Table \ref{table2}. 

Cases 8, 9, 12 and 14 can be excluded as $|O|$ must be even, while Cases 13, 17 and 18 cannot occur since $|O|-1$ is a prime power. 

Cases 10, 11, 15 and 16 can be excluded as follows. Note that $(|O|,p) =(12,11),(24,23),(28,3)$. 
If $(|O|,p)=(12,11)$ then $|G_P^{(1)}| < (1,3) \cdot (g-1)$ from Item 1 of Result \ref{list}, so that from Item 2 of Result \ref{list}, $|G_P| < (1,3) \cdot (g-1)(4g+4) < 5g^2$. From the Orbit stabilizer Theorem $|G| = 12|G_P|< 60g^2$, a contradiction.

If $(|O|,p)=(24,23)$ then $|G_P^{(1)}| < (1,1) \cdot (g-1)$ from Item 1 of Result \ref{list}, so that from Item 2 of Result \ref{list},  $|G_P| < (1,1) \cdot (g-1)(4g+4) <5g^2$. From the Orbit stabilizer Theorem $|G| = 24|G_P| < 120g^2$, a contradiction.

If $(|O|,p)=(28,3)$ then $|G_P^{(1)}| \leq 3 (g-1)$ from Item 1 of Result \ref{list}, so that from Item 2 of Result \ref{list},  $|G_P| \leq 3 (g-1)(4g+4)< 12g^2$. From the Orbit stabilizer Theorem $|G| = 28|G_P| < 336g^2$, a contradiction.

Since the stabilizer of a point in $A_n$ with $n \geq 5$ is not solvable we get that Case 1 in Table \ref{table2} cannot occur from Item 4 of Result \ref{list}.

Suppose that $G$ satisfies one of the remaining cases in Table \ref{table2}, that is, Cases 2-7. Let $\cY$ be the quotient curve $\cX/K$. We claim that either $K$ is trivial or $\cX/K$ is rational.

Let $K$ be trivial. Since $|G| \geq 900g^2$, $|G_P|>30(g-1)$ and Sylow $p$-subgroups in $G$ intersect trivially, we can apply Lemma \ref{terribile2} to get a contradiction.

Hence $K$ is not trivial. Suppose first that $g(\cX/K) \geq 2$. The Hurwitz genus formula implies that $2g-2 \geq |K|(2g(\cX/K)-2)$ so that
$$\frac{|G|}{|K|} \geq \frac{|G|(2g(\cX/K)-2)}{2g-2} > \frac{336g^2(g(\cX/K)-1)}{g-1} \geq 336g(\cX/K)^2.$$
We claim that the orbit $\bar O$ of $G/K$ lying below $O$ in $\cX|\cX/K$ is the only short orbit of $G/K$ on $\cX/K$.
Suppose by contradiction that $G/K$ has another short orbit $\bar O_n$. Then $G$ acts on set of points lying above $\bar O_n$, say $O_n$. Since $K$ has no other short orbits other than $O$, we have that $|O_n|=|K||\bar O_n| < |K| |G|/|K|=|G|$ so that also $O_n$ is a short orbit of $G$; a contradiction. This shows that if $g(\cX/K)\geq 2$ then $\cX/K$ has an automorphism group $G/K$ or order at least $336g(\cX/K)^2$ with exactly one short orbit on $\cX/K$. Since $g(\cX/K)<g$ this is not possible for the minimality of $g$. 

Thus, $g(\cX/K) \leq 1$. If $g(\cX/K)=1$ then, denoting with $\bar G_{\bar P}$ the stabilizer of the point $\bar P$ with $P|\bar P$ in $\cX|\cX/K$, $|\bar G_{\bar P}| \leq 12$ from \cite[Theorem 11.94]{HKT}.

We observe that in Cases 2-7 in Table \ref{table2}, $|\bar G_{\bar P}|\geq |O|$ so that also $|O| \leq 12$. We get $|G_P^{(1)}| \leq 3 (g-1)$ from  Item 1 of Result \ref{list}, and from Item 2 of Result \ref{list}, $|G_P| \leq 3(g-1)(4g+4) \leq 12g^2$. From the Orbit stabilizer Theorem $|G| \leq 12|G_P| \leq 144g^2$, a contradiction.
This shows that $\cX/K$ is rational. Since Cases 3-7 cannot give rise to automorphism groups of the rational function field from \cite[Theorem 11.14]{HKT}, only Case 2 can occur and so $\bar G=G/K \cong PSL(2,q),PGL(2,q)$. 

We note that $(q^d-1)/(q-1)-1=(q^d-q)/(q-1)$ cannot be a prime power unless $d=2$ and $q \geq 5$. 

Suppose first that $(|K|,p)=1$. Then Sylow $p$-subgroups of $G$ correspond to Sylow $p$-subgroups of either $PSL(2,q)$ or $PGL(2,q)$ and hence in any case they intersect trivially. Since $|G| >16g^2$ then claim follows from Lemma \ref{terribile2}.

Hence $K$ contains $p$-elements. Let $Q_K$ be a Sylow $p$-subgroup of $K$. Then $Q_K$ is a normal subgroup of $G$ implying that $G$ has normal $p$-subgroups. Denote by $S$ the largest normal $p$-subgroup of $G$ and consider the quotient curve $\cX/S$. From our choice of the subgroup $S$, the quotient group $G_S=G/S$ has no normal $p$-subgroups. We claim that $\cX/S$ is rational. 

Suppose first that $g(\cX/S) \geq 2$.  As before, the Hurwitz genus formula implies that $2g-2 \geq |S|(2g(\cX/S)-2)$ so that
$$\frac{|G|}{|S|} \geq \frac{|G|(2g(\cX/S)-2)}{2g-2} > \frac{336g^2(g(\cX/S)-1)}{g-1} \geq 336g(\cX/S)^2.$$
We claim that the orbit $\bar O$ of $G/S$ lying below $O$ in $\cX|\cX/S$ is the only short orbit of $G/S$ on $\cX/S$.
Suppose by contradiction that $G/S$ has another short orbit $\bar O_n$. Then $G$ acts on set of points lying above $\bar O_n$, say $O_n$. Since $K$ has no other short orbits other than $O$, we have that $|O_n|=|S||\bar O_n| < |S| |G|/|S|=|G|$ so that also $O_n$ is a short orbit of $G$; a contradiction. This shows that if $g(\cX/S)\geq 2$ then $\cX/S$ has an automorphism group $G/S$ of order at least $336g(\cX/S)^2$ with exactly one short orbit on $\cX/S$. Since $g(\cX/S)<g$ this is not possible for the minimality of $g$. 

If $g(\cX/S)=1$ then the group $G_PS/S$ is a subgroup of $Aut(\cX/S)$ fixing at least one point on $\cX/S$ (the one lying below $P$ in $\cX|\cX/S$). Hence $|U| \leq |G_P|/|S_P| =|G_{P}S|/|S|=|G_P|/|G_P \cap S| \leq 12$ from \cite[Theorem 11.94]{HKT}.
Recalling that $|O| \mid (2g-2)$ and that $|G_P^{(1)}| \leq p(g-1)/(p-2)$ from Item 1 of Result \ref{list}, we get $|G| \leq 12(2g-2)p(g-1)/(p-2)<336g^2$; a contradiction.
Hence $\cX/S$ is rational and $G/S$ is a subgroup of $PGL(2,\K)$ with no normal $p$-subgroups. From the classification of finite subgroups $PGL(2,\K)$, see \cite{VM}, $G/S$ is a prime to-p-subgroup which is either cyclic, or dihedral, or isomorphic to one of the the groups $A_4$, $S_4$. 
If $G/S \cong A_4,S_4$ then $|G| \leq 24|S||O| \leq 24(2g-2)p(g-1)/(p-2)<336g^2$ and hence we can discard these cases.
Since $\gamma  \equiv 0 \pmod p$, $S$ has exactly one fixed point (and possibly other non-trivial short orbits) on $\cX$ from Equation \eqref{eq2deuring}. Using the fact that $S$ is normal in $G$ we get that the entire $G$ has a fixed point on $\cX$ which is not possible as $|O|=q+1$.
\end{itemize}
\end{itemize}
\end{itemize}
\end{proof}

Now Items 1-3 in Theorem \ref{maint} are proven.
The next section will be devoted to the proof of Item 4 of Theorem \ref{maint}.

\section{$G$ satisfies Case $4$ of Theorem \ref{Hurwitz}}

In order to prove the main theorem we assume that $\cX$ is an algebraic curve defined over an algebraically closed field of odd characteristic $p$ and genus $g$. Let $G$ be an automorphism group of $\cX$. 
We assume that $|G| \geq 900g^2$ and that $G$ satisfies Case 4 of Theorem \ref{Hurwitz}, so that $G$ has two short orbits, one tame $O_2$ and one non-tame $O_1$. 
Furthermore, choosing a point $P$ in $O_1$, from  Item 4 of Result \ref{list}, we write $G_P = G_P^{(1)} \rtimes U$ with $|U|$ prime-to-$p$ and cyclic. Let $\mathcal{Y}_1$ be the quotient curve $\cX/{G_P^{(1)}}$. To complete the proof of the main theorem we can also assume that $g(\cY_1)=0$ and that $G_{P,R}$ is a $p$-group if it has a non-trivial Sylow $p$-subgroup.
The following three cases are treated separately. 
\begin{enumerate}
%\item[(iv.1)] $\cY_1$ is not rational. 
\item[(iv.1)] There is a point $R$ distinct from $P$ such that the stabilizer of $R$ in $G_P$ has order $p^t$ with $t \geq 1$. 
%\item[(iv.2)] There is a place $R$ distinct from $P$ such that the stabilizer of $R$ in $G_P$ has order $hp^t$ with $t \geq 1$, $h > 1$ and $h$ prime to $p$. 
\item[(iv.2)] No non-trivial element of $G_P^{(1)}$ fixes a place distinct from $P$, and there is a place $R$ distinct from $P$ but lying in the orbit of $P$ in $G$ such that the stabilizer of $R$ in $G_P$ is trivial. 
\item[(iv.3)] No non-trivial element of $G_P^{(1)}$ fixes a place distinct from $P$, and, for every place $R$ distinct from $P$ but lying in the orbit of $P$ in $G$, the stabilizer of $R$ in $G_P$ is non-trivial. 
\end{enumerate}

%\subsection{$G$ satisfies case (iv.1)}
%\textcolor{red}{FINIRE}

\subsection{$G$ satisfies Case (iv.1)}
In this case,  if $R$ is an arbitrary point on $\cX$ distinct from $P$ then the stabilizer of $R$ in $G_P$ is either trivial or a $p$-group of order $p^t$, $t \geq 1$.
Since $\cY_1$ is rational, $\bar U=G_P/G_P^{(1)} \cong U$ is cyclic it fixes two points on $\cY_1$ from  Item 5 of Result \ref{list}, say $\bar P$ (lying below $P$) and $\bar R$. Denote by $O_{\bar R}$ the orbit of $G_P^{(1)}$ lying above $\bar R$. Then $|O_{\bar R}|=p^i$ for some $i$ and $U$ acts on $O_R$. Since $(|U|,p)=1$, $U$ has at least one fixed point on $O_{\bar R}$, a contradiction. This shows that $|U|=1$. 
Assume by contradiction that $\gamma>0$.
From the Hurwitz genus formula applied to $G$ we have
$$2g-2=-2|G|+|O_1|d_P + |O_2|(|G_Q|-1)=-|G|+\frac{|G|}{|G_P|}d_P-\frac{|G|}{|G_Q|},$$
so that 
$$
|G|=\frac{(2g-2)|G^{(1)}_P||G_Q|}{-|G^{(1)}_P||G_Q|+d_P|G_Q|-|G^{(1)}_P|}.
$$

If $|G_Q| \leq 3$ then from Item 1 of Result \ref{list},
$$|G| \leq 6(g-1)|G^{(1)}_P| \leq \frac{6p(g-1)^2}{p-2} \leq 18g^2,$$
a contradiction.
So $|G_Q|>3$. Since $d_P \geq 2|G_P^{(1)}|-2$ and $|G_P^{(1)}| \geq 3$ we get 
$$-|G^{(1)}_P||G_Q|+d_P|G_Q|-|G^{(1)}_P| \geq -|G^{(1)}_P||G_Q|+(2|G_P^{(1)}|-2)|G_Q|-|G^{(1)}_P| $$
$$= |G^{(1)}_P||G_Q|-2|G_Q|-|G^{(1)}_P| \geq |G_Q|-3.$$
Thus, using again Item 1 of Result \ref{list},
$$|G|\leq \frac{(2g-2)|G^{(1)}_P||G_Q|}{|G_Q|-3} < 8g|G_P^{(1)}| \leq \frac{8pg(g-1)}{p-2} \leq 24g^2,$$
a contradiction.

\subsection{$G$ satisfies Case (iv.2)}
In this case the orbit $o(R)$ of $G_P$ containing $R$ is long. Let $o^\prime(R)$ be the orbit of $R$ under $G$. Then from the Orbit Stabilizer Theorem $|o^\prime(R)|\cdot|G_R|=|G|$. Moreover, as $R$ lies in the orbit of $P$ in $G$, also $o(R) \subseteq o^\prime(R)$. Let $Q$ be a place contained in the unique tame short orbit of $G$. From Equation \eqref{eq1} applied to $G$,
$$2(g-1)=-2|G|+\frac{|G|}{|G_Q|}(|G_Q|-1)+\frac{|G|}{|G_P|}d_P,$$
where $d_P$ denotes the ramification at $P$ and $d_Q=e_Q-1=|G_Q|-1$.
Hence
\begin{equation} \label{eqhurwitz}
|G|=2(g-1)\frac{|G_P^{(1)}||U|\cdot |G_Q|}{|G_Q|(d_P-|G_P|)-|G_P|}.
\end{equation}
Combining Equation \eqref{eqhurwitz} and $|o(R)|=|G_P| \leq |o^\prime(R)|$ yields $|G|=|G_P|o^\prime(R)| \geq |G_P|^2$, whence
$$|G_P| \leq \frac{|G|}{|G_P|}=2(g-1)\frac{|G_Q|}{|G_Q|(d_P-|G_P|)-|G_P|} \leq 2(g-1)|G_Q|.$$
Thus,
$$|G_Q|(d_P-|G_P|)-|G_P| \geq d_P|G_Q|-|G_P||G_Q|-2(g-1)|G_Q|,$$
and 
\begin{equation} \label{dis1}
|G_Q|(d_P-|G_P|)-|G_P| \geq |G_Q|(d_P-|G_P|-2(g-1)).
\end{equation}
 From Equation \eqref{eq1} applied to $G_P$,
$$2(g-1)=-2|G_P|+(d_P+|G_P^{(1)}|(|U|-1))=d_P-|G_P|-|G_P^{(1)}|,$$
and hence 
$$d_P-|G_P|-2(g-1)=|G_P^{(1)}|.$$
From Equation \eqref{dis1} and Result \ref{tameorder},
$$|G| \leq 2(g-1)\frac{|G_P^{(1)}||U||G_Q|}{|G_P^{(1)}||G_Q|} \leq 2(g-1)(4g+2)<8g^2.$$
Since $|G| >900g^2$, this case cannot occur.

\subsection{$G$ satisfies Case (iv.4)}

%Denote by $O_{NT}$ the unique non-tame short orbit of $G$. As $O_{NT}$ is a finite set, it is contained in $\cX(\mathbb{F}_{q^m})$ for some $m \geq 1$. We can assume without loss of generality that $m$ divides $n$ (up to replace $m$ with $mn$). Let $Q=q^m$, so that $G$ is defined over $\mathbb{F}_{Q}$ and $O_{NT}$ is contained in $\cX(\mathbb{F}_{Q})$. We note that $|\cX(\mathbb{F}_Q)|$ is congruent to $1$ modulo $|G_P^{(1)}|$. In fact, as $G_P^{(1)}$ satisfies case (iv.4), it fixes exactly one place $P \in \cX(\mathbb{F}_Q)$ acting with long orbits on $\cX(\mathbb{F}_Q) \setminus \{P\}$, that is, $|\cX(\mathbb{F}_{Q})|=k|G_P^{(1)}|+1$ for some $k \geq 0$. Let $p^t$ be the greatest power of $p$ dividing $|G|$. We note that $|\cX(\mathbb{F}_{Q^i}) \setminus \cX(\mathbb{F}_Q)|$ is congruent to $0$ modulo $p^t$ for every $i \geq 2$, as $\cX(\mathbb{F}_{Q^i}) \setminus \cX(\mathbb{F}_Q)$ is a union of long orbits or a union of long orbits with the unique tame short orbit of $G$. This implies that $|\cX(\mathbb{F}_{Q^i})|$ is congruent to $1$ modulo $|G_P^{(1)}|$ for every $i \geq 2$. 
%Hence every coefficient $a_n$ of the $L$-polynomial $L(t)=L_Q(t)$ of $\cX$ is zero, unless $|G_p^{(i)}|$ divides $n$.
%This gives that either $\gamma=0$ or $|G_P^{(1)}|$ divides $\gamma$. In any case $\gamma$ is congruent to zero modulo $p$.
%
%
%We conclude this section with the following open problem.
%
%\begin{openproblem}
%Is it true that every curve of genus $g \geq 2$ equipped with an automorphism group $G$ of order $|G|\geq 84g^2$ has zero $p$-rank?
%\end{openproblem}

Let as before $P \in O_1$. If $O_1=\{P\}$ then $G=G_P$. In particular $G$ is solvable and the claim follows from Theorem \ref{even}. Hence we can assume that $O \supset \{P\}$. We start with an intermediate lemma.
We assume that $\cX$ is a minimal counterexample with respect to the genus, that is, if $\cY$ is an algebraic curve of genus $\bar g < g$ together with an automorphism group $\bar G$ satisfying Case 4 in Theorem \ref{Hurwitz} then $|\bar G| <900 \bar g^2$. 

\begin{lemma} \label{pr}
If $O_1 \subset \{P\}$ then $O_1$ has size $q+1$ where $q=p^n$, $n \geq 1$ and $G$ acts $2$-transitively on $O_1$. If $K$ denotes the Kernel of the permutation representation of $G$ over $O_1$, one of the following cases occurs.
\begin{itemize}
\item $G/K \cong PGL(2,q),PSL(2,q)$,
\item $G/K \cong PGU(3,q), PSU(3,q)$,
\item $G/K \cong Ree(q)$, when $p=3$, $q=3^{2t+1}$,
\item $G/K$ has a regular normal soluble subgroup and the size of $O_1$ is a prime power.
\end{itemize}
Unless the last case occur, if $|G_P| \geq 30(g-1)$ then $\cX$ has zero $p$-rank.
\end{lemma}

\begin{proof}
Let $o_0=\{P\}, o_1,\ldots,o_k$ denote the orbits of $G_P^{(1)}$ contained in $O_1$, so that $O_1=\bigcup_{i=0}^{k} o_i$. To prove that $G$ acts $2$-transitively on $O_1$ we show that $k=1$. 

For any $i=1,\ldots,k$ take $R_i \in o_i$. Since we are dealing with Case (iv.4), $R_i$ is fixed by an element $\alpha_i \in G_P$ of prime order $m \ne p$ dividing $|U|$. By Sylow's Theorem there exist a subgroup $U^i$ conjugated to $U$ in $G_P$ containing $\alpha_i$ and $\alpha_i$ clearly preserves $o_i$. As previously noted, since $\cY_1=\cX/G_P^{(1)}$ is rational, $\alpha_i$ fixes at most two $G_P^{(1)}$-orbits and hence $o_0$ and $o_i$ are the only fixed orbits of $\alpha_i$. Since $U^i$ is abelian and it fixes $o_0$, the orbits $o_0$ and $o_i$ are also the only $G_P^{(1)}$-orbits fixed by $U^i$. Since we can write $G_P=G_P^{(1)} \rtimes U^i$ we get that the whole $G_P$ fixes $o_i$ for all $1\leq i \leq k$. Thus, either $k=1$ or $G_P$ fixes at least $3$ $G_P^{(1)}$-orbits. The latter case cannot occur from  Item 5 of Result \ref{list} applied to $\cY_1$ as automorphisms of a curve of genus zero have at most $2$ fixed points.

This shows that $k=1$ so that $G$ acts $2$-transitively on $O_1$. Also $|O_1|=q+1$ with $q=|G_P^{(1)}|=p^n$, $n \geq 1$ as we are dealing with Case (iv.4). Let $K$ denote the kernel of the action of $G$ on $O_1$ and let $\bar G=G/K$ be the corresponding permutation group. Since the stabilizer of $2$ points in $G$ (and hence $\bar G$) is cyclic and $p$ is odd we get from Theorem \ref{KOS2t} that $\bar G$ is isomorphic to one of the groups listed.
In the last case $O_1$ is a prime power and the regular normal subgroup an elementary abelian group from Theorem \ref{BWrn}.

From $(|K|,p)=1$, we have that Sylow $p$-subgroups of $G$ corresponds to Sylow $p$-subgroups of $\bar G$. Since in all the cases but the last one Sylow $p$-subgroups of $\bar G$ intersect trivially the same holds for $G$. If $|G_P| \geq 30(g-1)$ the claim follows from Lemma \ref{terribile2}.
\end{proof}

Assume the one of the first $3$ cases listed in Lemma \ref{pr} occurs with $|G_P| <30(g-1)$. First of all we note that $K$ is not trivial. Suppose indeed by contradiction that $K$ is trivial so that $G\cong \bar G$.
\begin{itemize}
\item $G \cong PSL(2,q),PGL(2,q)$. Here $|G_P|=q(q-1)/2$ or $|G_P|=q(q-1)$ and in any case $|G_P|^2>|G|=(q+1)|G_P| \geq 900g^2$. Hence $|G_P|>30(g-1)$, a contradiction.
\item $G \cong PSU(3,q),PGU(2,q)$. In this case $|G_P|=q^3(q^2-1)/3$ or $|G_P|=q^3(q^2-1)$ and in any case $|G_P|^2>|G|=(q^3+1)|G_P|\geq 900g^2$. Hence $|G_P|>30(g-1)$, a contradiction.
\item $G \cong Ree(q)$. Now, $|G_P|=q^3(q-1)$ and $|G_P|^2>|G|=(q^3+1)|G_P|\geq 900g^2$. Hence $|G_P|>30(g-1)$, a contradiction.
\end{itemize}

Before analyzing the case in which $K$ is not trivial, we prove a trivial intersection condition for the stabilizers of points in distinct $G$-orbits.

\begin{lemma} \label{lem2}
If $Q \in O_2$ then $G_P \cap G_Q$ is trivial.
\end{lemma}

\begin{proof}
Let $\alpha \in G_P \cap G_Q$ non-trivial. Then the order of $\alpha$ is not divisible by $p$. Hence $\alpha \in U$ up to conjugation. Since $|O_1|=q+1$ from Lemma \ref{pr}, $\alpha$ fixes at least another point $R \in O_1 \setminus \{P\}$. Since $P$, $Q$ and $R$ are in three distinct $G_P^{(1)}$-orbit the automorphism $\bar \alpha$ induced by $\alpha$ on $\cY_1$ has at least three fixed point on $\cY_1$. Since the order of $\alpha$ is the same as the order of $\bar \alpha$, from Item 5 of Result \ref{list}  we get that $\alpha$ has order 1, completing the proof.
\end{proof}

This proves that $K$ is not trivial and from Lemma \ref{lem2} the only short orbits of $K$ are exactly the points in $O_1$.

 We claim that $g(\cX/K)=0$. Assume first that $g(\cX/K) \geq 2$. Then arguing as for the previous cases using the Hurwitz genus formula, $|G|/|K| \geq 900g(\cX/K)^2$. Also the set of points $\bar O_1$ and $\bar O_2$ lying below $O_1$ and $O_2$ in $\cX/\cX/K$ are respectively a short and a long orbit of $G/K$. If $G/K$ has not exactly one another tame short orbit on $\cX/K$ then $|G|/|K| \leq 336g(\cX/K)^2$ from Section 3 and Theorem \ref{maint} Item 1, a contradiction. Since now $G_/K$ has exactly two short orbits (one tame and one non-tame) and $g(\cX/2)<g$ from the minimality of $g$ we get a contradiction.

Suppose that $g(\cX/K)=1$. Then $|U| \leq 12$ from \cite[Theorem 11.94]{HKT}. Note that if $\gamma>0$ then from Item 1 of Result \ref{list}, $|O|=1+|G_P^{(1)}| <2|G_P^{(1)}| \leq 2p(g-1)/(p-2) <4g$. Hence
$$|G| < 4g|G_P^{(1)}||U| \leq 192g^2,$$
a contradiction.
Thus $\cX/K$ is rational and hence $G/K \cong PSL(2,q), PGL(2,q)$ since the other groups do not occur as subgroups of automorphisms of curves of genus zero. From the Hurwitz genus formula
$$2g-2=-2|K|+(|K|-1)(q+1),$$
so that $g=(q-1)(|K|-1)/2$. In particular $|O||K|=(q+1)|K| \leq 10g$. Hence
$$900g^2 \leq |G|=|O||G_P| < 10g|G_P|.$$
So, $|G_P| >30(g-1)$. The claim now follows from Lemma \ref{terribile2}.

We are left with the last case in Lemma \ref{pr}, that is, a minimal normal subgroup of $G/K$ is soluble and the size of $O_1$ is a prime power. Since $q$ is odd, we get $q+1=2^t$ for some $t\geq 1$. 
From Mihăilescu Theorem either $q=p$ is a Marsenne prime or $q=8$. Since $q$ is odd, $q=p$ is a Marsenne prime and $q=p=|G_P^{(1)}|$.
Assume that $\gamma \ne 0$.
If $|G_P|<30(g-1)$ we get from Item 1 of Result \ref{list} that $|G|=|O||G_P|<30(g-1)(|G_P^{(1)}|+1)<900g^2$, a contradiction.
Hence $|G_P| \geq 30(g-1)$ and Lemma \ref{terribile2} gives the desired contradiction. Indeed Sylow $p$-subgroup of $G$ intersect trivially as they all fix exactly one point on $O_1$.

This proves Item 4 in Theorem \ref{maint} so that Theorem \ref{maint} is completely proven.
We conclude this section with the following open problem.
\begin{openproblem}
Is the condition $|G| \geq 900g^2$ sufficient to imply $\gamma(\cX)=0$ also when $g(\cX/G_P^{(1)}) \geq 1$ for $P \in O_1$, or $|G_{P,R}|=p^t h$ with $R \in O_1 \setminus \{P\}$, $t \geq 1$ and $h>1$?
\end{openproblem}

\section*{Acknowledgments}
The author would like to thank Prof. Massimo Giulietti and Dr. Pietro Speziali for the helpful discussions and comments on the topic.

\end{document}